\documentclass[11pt,a4paper]{article}

\usepackage{amssymb} 
\usepackage{amsmath} 
\usepackage{amsthm} 
\usepackage{hyperref}
\usepackage{todonotes}

\usepackage{a4wide}
\usepackage{tikz}
\newtheorem{thr}{Theorem}
\newtheorem{conjecture}[thr]{Conjecture}
\newtheorem{claim}[thr]{Claim}

\newtheorem{prop}[thr]{Proposition}

\theoremstyle{definition}

\newtheorem{cor}[thr]{Corollary}

\newcommand*{\myproofname}{Proof}
\newenvironment{claimproof}[1][\myproofname]{\begin{proof}[#1]}{\end{proof}}

\newcommand{\eps}{\varepsilon}

\title{Maximising line subgraphs of diameter at most $t$\footnote{A preliminary version of this paper appeared as an extended abstract in {\em Proceedings of the 11th European Conference on Combinatorics, Graph Theory and Applications (EuroComb 2021, Barcelona)}, {\em Trends in Mathematics} 14: 331-338, 2021. \url{https://doi.org/10.1007/978-3-030-83823-2_52}}}

\author{
	Stijn Cambie
	\thanks{Department of Mathematics, . 
		Email: \protect\href{mailto:stijn.cambie@hotmail.com}{\protect\nolinkurl{stijn.cambie@hotmail.com}}, \protect\href{mailto:r.deverclos@math.ru.nl}{\protect\nolinkurl{r.deverclos@math.ru.nl}}, \protect\href{mailto:ross.kang@gmail.com}{\protect\nolinkurl{ross.kang@gmail.com}}. Supported by a Vidi grant (639.032.614) of the Netherlands Organisation for Scientific Research (NWO).}
	\and
	Wouter Cames van Batenburg
	\thanks{Laboratoire G-SCOP (CNRS, Univ. Grenoble Alpes), Grenoble, France. 
		Email: \protect\href{mailto:wouter.cames-van-batenburg@grenoble-inp.fr}{\protect\nolinkurl{wouter.cames-van-batenburg@grenoble-inp.fr}}. Partially supported by ANR Project GATO
(\textsc{ANR-16-CE40-0009-01})}
	\and
	R\'emi de Joannis de Verclos
	\footnotemark[2]
	\and
	Ross J. Kang
	\footnotemark[2]
}

\begin{document}
		\definecolor{xdxdff}{rgb}{0.49019607843137253,0.49019607843137253,1.}
	\definecolor{ududff}{rgb}{0.30196078431372547,0.30196078431372547,1.}
	
	\tikzstyle{every node}=[circle, draw, fill=black!50,
	inner sep=0pt, minimum width=4pt]

\maketitle

\begin{abstract}
We wish to bring attention to a natural but slightly hidden problem, posed by Erd\H{o}s and Ne\v{s}et\v{r}il in the late 1980s, an edge version of the degree--diameter problem. Our main result is that, for any graph of maximum degree $\Delta$ with more than $1.5 \Delta^t$ edges, its line graph must have diameter larger than $t$. In the case where the graph contains no cycle of length $2t+1$, we can improve the bound on the number of edges to one that is exact for $t\in\{1,2,3,4,6\}$. In the case $\Delta=3$ and $t=3$, we obtain an exact bound. Our results also have implications for the related problem of bounding the distance-$t$ chromatic index, $t>2$; in particular, for this we obtain an upper bound of $1.941\Delta^t$ for graphs of large enough maximum degree $\Delta$, markedly improving upon earlier bounds for this parameter.

\end{abstract}


\section{Introduction}\label{sec:intro}

Erd\H{o}s in~\cite{Erd88} wrote about a problem he proposed with Ne\v{s}et\v{r}il:
\begin{quote}
``One could perhaps try to determine the smallest integer $h_t(\Delta)$ so that every $G$ of $h_t(\Delta)$ edges each vertex of which has degree $\le \Delta$ contains two edges so that the shortest path joining these edges has length $\ge t$ \dots {\em This problem seems to be interesting only if there is a nice expression for $h_t(\Delta)$}.''
\end{quote}
Equivalently, $h_t(\Delta)-1$ is the largest number of edges inducing a graph of maximum degree $\Delta$ whose line graph has diameter at most $t$.
Alternatively, one could consider this an edge version of the (old, well-studied, and exceptionally difficult) degree--diameter problem, cf.~\cite{B78}.

It is easy to see that $h_t(\Delta)$ is at most $2\Delta^t$ always, but one might imagine it to be smaller.
For instance, the $t=1$ case is easy and $h_1(\Delta)=\Delta+1$.
For $t=2$, it was independently proposed by Erd\H{o}s and Ne\v{s}et\v{r}il~\cite{Erd88} and Bermond, Bond, Paoli and Peyrat~\cite{BPPT83} that
$h_2(\Delta) \le 5\Delta^2/4+1$, there being equality for even $\Delta$. This was confirmed by Chung, Gy\'arf\'as, Tuza and Trotter~\cite{CGTT90}. 
For the case $t=3$, we suggest the following as a ``nice expression''.
\begin{conjecture}\label{conj:CJK}
$h_3(\Delta)\le \Delta^3-\Delta^2+\Delta+2$, with equality if $\Delta$ is one more than a prime power.
\end{conjecture}
\noindent
As to the hypothetical sharpness of this conjecture, first consider the point--line incidence graphs of finite projective planes of prime power order $q$. Writing $\Delta=q+1$, such graphs are bipartite, $\Delta$-regular, and of girth $6$; their line graphs have diameter $3$; and they have $\Delta^3-\Delta^2+\Delta$ edges. At the expense of bipartiteness and $\Delta$-regularity, one can improve on the number of edges in this construction by {\em one} by subdividing one edge, which yields the expression in Conjecture~\ref{conj:CJK}.
We remark that for multigraphs instead of simple graphs, one can further increase the number of edges by $\left \lfloor\frac \Delta2 \right \rfloor-1,$ by deleting some arbitrary vertex $v$ and replacing it with a multiedge of multiplicity $\left \lfloor\frac \Delta2 \right \rfloor,$ whose endvertices are connected with $\left \lfloor\frac \Delta2 \right \rfloor$ and $\left \lceil \frac \Delta2 \right \rceil$ of the original $\Delta$ neighbours of $v.$
This last remark contrasts to what we know for multigraphs in the case $t=2$, cf.~\cite{CaKa19,JKP19}.

Through a brief case analysis, we have confirmed Conjecture~\ref{conj:CJK} in the case $\Delta=3$.
\begin{thr}\label{thm:cubic}
The line graph of any (multi)graph of maximum degree $3$ with at least $23$ edges has diameter greater than $3$.
That is, $h_3(3) = 23$.
\end{thr}

For larger fixed $t$, although we are slightly less confident as to what a ``nice expression'' for $h_t(\Delta)$ might be, we believe that $h_t(\Delta)=(1+o(1))\Delta^t$ holds for infinitely many $\Delta$. 

We contend that this naturally divides into two distinct challenges, the former of which appears to be more difficult than the latter.
\begin{conjecture}\label{conj:lower}
For any $\eps>0$, $h_t(\Delta) \ge (1-\eps)\Delta^t$ for infinitely many $\Delta$.
\end{conjecture}

\begin{conjecture}\label{conj:upper}
For $t\ne 2$ and any $\eps>0$, $h_t(\Delta) \le (1+\eps)\Delta^t$ for all large enough $\Delta$.
\end{conjecture}

\noindent
With respect to Conjecture~\ref{conj:lower}, we mentioned earlier how it is known to hold for $t\in\{1,2,3\}$. For $t\in\{4,6\}$, it holds due to the point--line incidence graphs of, respectively, a symplectic quadrangle with parameters $(\Delta-1,\Delta-1)$ and a split Cayley hexagon with parameters $(\Delta-1,\Delta-1)$ when $\Delta-1=q$ is a prime power. For all other values of $t$ the conjecture remains open. 
Conjecture~\ref{conj:lower} may be viewed as the direct edge analogue of an old conjecture of Bollob\'as~\cite{B78}. That conjecture asserts, for any positive integer $t$ and any $\eps>0$, that there is a graph of maximum degree $\Delta$ with at least $(1-\eps)\Delta^t$ vertices of diameter at most $t$ for infinitely many $\Delta$.
The current status of Conjecture~\ref{conj:lower} is essentially the same as for Bollob\'as's conjecture: it is unknown if there is an absolute constant $c>0$ such that $h_t(\Delta) \ge c \Delta^t$ for all $t$ and infinitely many $\Delta$.
For large $t$ the best constructions we are aware of are (ones derived from) the best constructions for Bollob\'as's conjecture.
\begin{prop}\label{prop:CaGo05}
There is $t_0$ such that $h_t(\Delta) \ge 0.629^t \Delta^t$ for $t\ge t_0$ and infinitely many $\Delta$.
\end{prop}
\begin{proof}
Canale and G\'omez~\cite{CaGo05} proved the existence of graphs of maximum degree $\Delta$, of diameter $t'$, and with more than $(0.6291\Delta)^{t'}$ vertices, for $t'$ large enough and infinitely many $\Delta$. Consider this construction for $t'=t-1$ and each valid $\Delta$.
Now in an iterative process arbitrarily add edges between vertices of degree less than $\Delta$.
Note that as long as there are at least $\Delta+1$ such vertices, then for every one there is at least one other to which it is not adjacent.
Thus by the end of this process, at most $\Delta$ vertices have degree smaller than $\Delta$, and so the resulting graph has at least $\frac 12 ( (0.6291\Delta)^{t'}\Delta  - \Delta^2 )$ edges, which is greater than $(0.629\Delta)^{t}$ for $t$ sufficiently large.
Furthermore since the graph has diameter at most $t-1$, its line graph has diameter at most $t$.
\end{proof}

\noindent
By the above argument (which was noted in~\cite{DeSl19eurocomb}), the truth of Bollob\'as's conjecture would imply a slightly weaker form of Conjecture~\ref{conj:lower}, that is, with a leading asymptotic factor of $1/2$.
As far as we are aware, a reverse implication, i.e.~from Conjecture~\ref{conj:lower} to some form of Bollob\'as's conjecture is not known.

Our main result is partial progress towards Conjecture~\ref{conj:upper} (and thus Conjecture~\ref{conj:CJK}).

\begin{thr}\label{thm:mainht}
$h_t(\Delta) \le \frac{3}{2}\Delta^t+1$.
\end{thr}

\noindent
Theorem~\ref{thm:mainht} is a result/proof valid for all $t\ge1$, but as we already mentioned there are better, sharp determinations for $t\in\{1,2\}$.
 We have also settled Conjecture~\ref{conj:upper} in the special case of graphs containing no cycle $C_{2t+1}$ of length $2t+1$ as a subgraph.

\begin{thr}\label{thm:2t+1ht}
The line graph of any $C_{2t+1}$-free graph of maximum degree $\Delta$ with at least $\Delta^t$ edges has diameter greater than $t$.
\end{thr}

\noindent
For $t\in\{1,2,3,4,6\}$, this last statement is asymptotically sharp (and in its more precise formulation the result is in fact exactly sharp) due to the point--line incidence graphs of generalised polygons.
The cases $t\in\{3,4,6\}$ are perhaps most enticing in Conjecture~\ref{conj:upper}, and that is why we highlighted the case $t=3$ first in Conjecture~\ref{conj:CJK}.

In order to discuss one consequence of our work, we can reframe the problem of estimating $h_t(\Delta)$ in stronger terms. Let us write $L(G)$ for the line graph of $G$ and $H^t$ for the $t$-th power of $H$ (where we join pairs of distinct vertices at distance at most $t$ in $H$). Then the problem of Erd\H{o}s and Ne\v{s}et\v{r}il framed at the beginning of the paper is equivalent to seeking optimal bounds on $|L(G)|$ subject to $G$ having maximum degree $\Delta$ and $L(G)^t$ inducing a clique. 
Letting $\omega(H)$ denote the clique number of $H$, our main results are proven in terms of bounds on the {\em distance-$t$ edge-clique number} $\omega(L(G)^t)$ for graphs $G$ of maximum degree $\Delta$. In particular, we prove  Theorem~\ref{thm:mainht} by showing the following stronger form.

\begin{thr}\label{thm:mainclique}
For any graph $G$ of maximum degree $\Delta$, it holds that $\omega(L(G)^t) \le \frac32\Delta^t$.
\end{thr}

\noindent
We should remark that D\c{e}bski and \'Sleszy\'nska-Nowak~\cite{DeSl19eurocomb} announced a bound of roughly $\frac74\Delta^t$.
Note that the bound in Theorem~\ref{thm:mainclique} can be improved in the cases $t\in\{1,2\}$: $\omega(L(G)) \le \Delta+1$ is trivially true, while $\omega(L(G)^2) \le \frac43\Delta^2$ is a recent result of Faron and Postle~\cite{FaPo19}.
We also have a bound on $\omega(L(G)^t)$ analogous to Theorem~\ref{thm:2t+1ht}, a result stated and shown in Section~\ref{sec:2t+1Free}.

A special motivation for us is a further strengthened form of the problem. In particular, there has been considerable interest in $\chi(L(G)^t)$ (where $\chi(H)$ denotes the chromatic number of $H$), especially for $G$ of bounded maximum degree. For $t=1$, this is the usual chromatic index of $G$; for $t=2$, it is known as the {\em strong chromatic index} of $G$, and is associated with a more famous problem of Erd\H{o}s and Ne\v{s}et\v{r}il~\cite{Erd88}; for $t>2$, the parameter is referred to as the {\em distance-$t$ chromatic index}, with the study of bounded degree graphs initiated in~\cite{KaMa12}.
We note that the output of Theorem~\ref{thm:mainclique} may be directly used as input to a recent result~\cite{HJK21} related to Reed's conjecture~\cite{Ree98} to bound $\chi(L(G)^t)$. This yields the following.

\begin{cor}\label{cor:fromreed}
There is some $\Delta_0$ such that, for any graph $G$ of maximum degree $\Delta\ge \Delta_0$, it holds that $\chi(L(G)^t) < 1.941\Delta^t$.
\end{cor}

\begin{proof}
By Theorem~\ref{thm:mainclique} and~\cite[Thm.~1.6]{HJK21}, $\chi(L(G)^t) \le \left \lceil0.881(\Delta(L(G)^t)+1)+0.119\omega(L(G)^t\right \rceil \le \left \lceil0.881(2\Delta^t+1)+0.119\cdot1.5\Delta^t\right \rceil < 1.941\Delta^t$ provided $\Delta$ is taken large enough.
\end{proof}

\noindent
For $t=1$, Vizing's theorem states that $\chi(L(G)) \le \Delta+1$.
For $t=2$, the current best bound on the strong chromatic index~\cite{HJK21} is $\chi(L(G)^2) \le 1.772\Delta^2$ for all sufficiently large $\Delta$.
For $t>2$, note for comparison with Corollary~\ref{cor:fromreed} that the local edge density estimates for $L(G)^t$ proved in~\cite{KaKa14} combined with the most up-to-date colouring bounds for graphs of bounded local edge density~\cite{HJK21} yields only a bound of $1.999\Delta^t$ for all large enough $\Delta$.
We must say though that, for the best upper bounds on $\chi(L(G)^t)$, $t>2$, rather than bounding $\omega(L(G)^t)$ it looks more promising to pursue optimal bounds for the local edge density of $L(G)^t$, particularly for $t\in\{3,4,6\}$. We have left this to future study.

\subsection{Terminology and notation}

For a graph $G=(V,E)$, we denote the $i^{th}$ neighbourhood of a vertex $v$ by $N_i(v)$, that is, $N_i(v)=\{u \in V \mid d(u,v)=i\}$, where $d(u,v)$ denotes the distance between $u$ and $v$ in $G$.
Similarly, we define $N_i(e)$ as the set of vertices at distance $i$ from an endpoint of $e$.

Let $T_{k,\Delta}$ denote a tree rooted at $v$ of height $k$ (i.e.~the leafs are exactly $N_k(v)$) such that all non-leaf vertices have degree $\Delta$.
Let $T^1_{k,\Delta}$ be one of the $\Delta$ subtrees starting at $v$, i.e.~a subtree rooted at $v$ of height $k$ such that $v$ has degree $1$, such that $N_k(v)$ only contains leaves and all non-leaf vertices have degree $\Delta$.

\section{A bound on $\omega(L(G)^t)$ for $C_{2t+1}$-free $G$}\label{sec:2t+1Free}

In this section, we prove the following theorem.

\begin{thr}\label{thm:L^tforC_{2t+1}-free}
	Let $t \ge 2$ be an integer. Let $G$ be a $C_{2t+1}$-free graph with maximum degree $\Delta.$
	Then $\omega(L(G)^t) \le \lvert E(T_{t,\Delta}) \rvert$.
	When $t \in \{2,3,4,6\}$ equality can occur for infinitely many $\Delta$.
\end{thr}

\noindent
Since $\lvert E(T_{t,\Delta}) \rvert \le \Delta^t$, the expression is at most the bound desired for Conjecture~\ref{conj:upper}, and thus this implies Theorem~\ref{thm:2t+1ht}. In fact, the expression matches the order of the point--line incidence graphs of generalised polygons when $t \in \{2,3,4,6\}$, which are the examples for which equality holds.
On the other hand, by subdividing one edge of any these constructions, one can see in the cases $t\in\{2,3,4,6\}$ that the result fails if we omit the condition of $C_{2t+1}$-freeness.
We note that Theorem~\ref{thm:L^tforC_{2t+1}-free} is a generalisation of a result in~\cite{CKP20} which was specific to the case $t=2$.
It is also a stronger form of a result announced in~\cite{DeSl19eurocomb} for bipartite graphs.
One might wonder about excluding other cycle lengths, particularly even ones. Implicitly this was already studied in~\cite{KaPi18}, in that the local sparsity estimations there imply the following statement: for any $t\ge 2$ and even $\ell\ge 2t$, $\omega(L(G)^t) = o(\Delta^t)$ for any $C_\ell$-free graph of maximum degree $\Delta$.
Similarly, it would be natural to pursue a similar bound as in Theorem~\ref{thm:L^tforC_{2t+1}-free} but for an excluded odd cycle length (greater than $2t+1$), which was done for $t=2$ in~\cite{CKP20}.

The bound in Theorem~\ref{thm:L^tforC_{2t+1}-free} is a corollary of the following proposition.

\begin{prop}\label{sub}
	For fixed $\Delta$ and $t$, let $G$ be a $C_{2t+1}$-free graph with maximum degree $\Delta$ and $H\subseteq G$ be a subgraph of $G$ with maximum degree $\Delta_H$.
	Let $v$ be a vertex with degree $d_H(v)=\Delta_H=j$ and let $u_1,u_2, \ldots, u_j$ be its neighbours. Suppose that in $L(G)^t$, every edge of $H$ is adjacent to $vu_i$ for every $1 \le i \le j$.
	Then $\lvert E(H) \rvert \le \lvert E(T_{t,\Delta}) \rvert$.	
\end{prop}

\begin{proof}
	For fixed $\Delta$, let $H$ and $G$ be graphs satisfying all conditions, such that $\lvert E(H) \rvert$ is maximised. This can be done since $\lvert E(H) \rvert$ is upper bounded by say $j\Delta ^t$.	
	With respect to the graph $G$, we write $N_i=N_i(v)$ for $0 \le i \le t+1.$
	We start proving a claim that makes work easier afterwards.
	\begin{claim}\label{cl:no_edge_inNi}
		For any $1 \le i \le t$, $H$ does not contain any edge between two vertices of $N_i$.
	\end{claim}
	
	\begin{claimproof}
		Suppose it is not true for some $i \le t-1$.
		Take an edge $yz \in E(H)$ with $y,z \in N_i$.
		Construct the graph $H'$ with $V(H')=V(H) \cup \{y',z'\}$ and $E(H')=E(H) \setminus yz \cup \{yy',zz'\}$, where $y'$ and $z'$ are new vertices, and let $G'$ be the corresponding modification of $G$.
		Then $H' \subseteq G'$ also satisfies all conditions in Proposition~\ref{sub} and $\lvert E(H') \rvert=\lvert E(H) \rvert+1,$ contradictory with the choice of $H$.
		
		Next, suppose there is an edge $yz \in E(H)$ with $y,z \in N_t$.
Take a shortest path from $u_1$ to $yz$, which is wlog a path $P_y$ from $u_1$ to $y$. Note that a shortest path $P_z$ from $v$ to $z$ will intersect $P_y$ since $G$ is $C_{2t+1}$-free. Let $w$ be the vertex in $V(P_y)\cap V(P_z)$ that minimises $d_G(w,z)$ and assume $w \in N_m,$ i.e.~$m=d_G(v,w)$ is the distance from $v$ to $w$.
The condition $w \in V(P_y)\cap V(P_z)$ ensures that $d_G(w,z)=d_G(w,y)$.	
		Furthermore note that $y$ and $z$ are interchangeable at this point, as both are at the same distance from $u_1$.

		If $d_G(w,u_i)=m-1$ for every $1 \le i \le j$, we can remove $yz$ again and add two edges $yy'$ and $zz'$ to get a graph $H'$ satisfying all conditions, leading to a contradiction again. This is the blue scenario illustrated in Figure~\ref{fig:CloseLookAtG&H}.
		
		In the other case there is some $1 <s \le j$ such that $d_G(w,u_s)>m-1$. Since $d_G(u_s,yz)=t-1$, wlog $d_G(u_s,z)=t-1$, there is a shortest path from $u_s$ to $z$ which is disjoint from the previously selected shortest path $P_y$ between $u_1$ and $y$. Hence together with the edges $u_1v, vu_s$ and $yz$, this forms a $C_{2t+1}$ in $G$, which again is a contradiction. This is sketched as the red scenario in Figure~\ref{fig:CloseLookAtG&H}.
	\end{claimproof}

\begin{figure}[h]
	\centering
	\begin{tikzpicture}[line cap=round,line join=roundx,x=1.3cm,y=1.3cm]
	\clip(-1,-0.7) rectangle (11,2.4);
	\draw [line width=1.1pt] (0,0.) -- (2.,2);
	\draw [line width=1.1pt] (0,0.) -- (2.,1.5);
	\draw [line width=1.1pt] (0,0.) -- (2.,0.);	
	\draw [dashed, line width=1.1pt] (0,0.) -- (2.,-0.5);
	
	\draw [dotted, line width=0.8pt] (2,0.) -- (2.,1.5);

	\draw[line width=1.1pt] (2,2)--(10,2) ;
	\draw[line width=1.1pt] (10,1)--(10,2) ;
	\draw[line width=1.1pt] (6,2)--(10,1) ;
		\draw[line width=1.1pt] (0,0)--(2,0.8) ;
	\draw[ultra thick, color=red] (7,1.75)--(2,0.8) ;
	\draw[ultra thick, color=blue] (6,2)--(2,0.8) ;
			\draw[ultra thick, color=blue] (6,2)--(2,1.5) ;
				\draw[ultra thick, color=blue] (6,2)--(2,0.) ;

	\draw [fill=xdxdff] (0,0.) circle (2.5pt);
	\draw [fill=xdxdff] (2.,0.) circle (2.5pt);
	\draw [fill=xdxdff] (2.,1.5) circle (2.5pt);
	\draw [fill=xdxdff] (2.,-.5) circle (2.5pt);
	\draw [fill=xdxdff] (2.,2.) circle (2.5pt);
	\draw [fill=xdxdff] (2,0.8) circle (2.5pt);
	
	\draw [fill=xdxdff] (10.,2.) circle (2.5pt);
	\draw [fill=xdxdff] (10.,1) circle (2.5pt);	
	\draw [fill=xdxdff] (6.,2.) circle (2.5pt);

	\coordinate [label=right:$y$] (A) at (10.05,2);
	\coordinate [label=right:$z$] (A) at (10.05,1);
	\coordinate [label=above:$w$] (A) at (6,2.05);
	
	\coordinate [label=left:$v$] (A) at (-0.1,0); 
	\coordinate [label=above left:$u_{1}$] (A) at (1.9,2.05); 
	\coordinate [label=above:$u_{2}$] (A) at (2,1.5); 
	
		\coordinate [label=left :$u_{s}$] (A) at (1.95,0.9); 

	\coordinate [label=above left:$u_{j}$] (A) at (1.9,0.05);

	\end{tikzpicture}
	\caption{Sketch of two scenarios (red and blue) referred to in the proof of Claim~\ref{cl:no_edge_inNi}.}
	\label{fig:CloseLookAtG&H}
\end{figure}
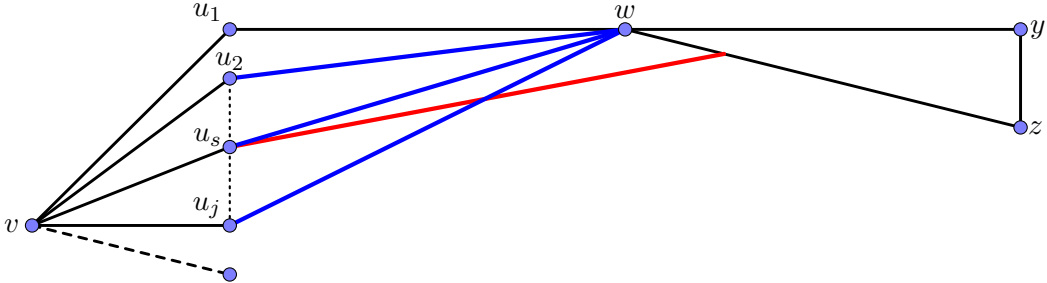

	For every $1 \le m \le t+1$, let $A_m$ be the set of all vertices $x$ in $N_m$ such that $d_G(v,x)=d_G(u_i,x)+1=m$ for at least one index $1 \le i \le j$ and let $R_m=N_m \backslash A_m.$
	Also let $A_0=\{v\}.$
	Let $A=\cup_{i=0}^{t+1} A_i$ and $R=\cup_{i=0}^{t+1} R_i$.
	We observe that a vertex in $R_{t+1}$ cannot be an endvertex of an edge of $H$. Indeed, assuming the contrary, the other endvertex of such an edge would be at distance $t-1$ from every $u_i$, $1 \le i \le j$ and in particular would belong to $A_t$, leading to a contradiction.
	Also we observe that there are no edges in $H$ between $R_t$ and $A$. By definition of $R_t$, any vertex $r \in R_t$ has no neighbour in $A_i$ where $i<t$, nor does it have a neighbour in $A_t$ by Claim~\ref{cl:no_edge_inNi}. To end with, an edge with endvertices in $R_t$ and $A_{t+1}$ is not connected to any edge $vu_s$, $1 \le s \le j$, in $L(G)^t$.
	As a consequence, the number of edges of $H$ which contain at least one vertex of $R$ can be upper bounded by $	|R_1| \cdot \left(|E(T_{t,\Delta}^{1})| -1\right)$, which equals \begin{equation}\label{eq:EH_R}
(\deg(v)-\Delta_H)\left(\frac1{\Delta}\lvert E(T_{t,\Delta}) \rvert -1 \right).
	\end{equation}
	
All other edges of $H$ are in the induced subgraph $H[A]$. We will now compute a bound on the number of those remaining edges.

We start with defining a weight function $w$ on the vertices $x$ in $A$ which will turn out to be useful.
For every $x \in A_m$ where $1 \le m \le t$, we define $w(x)$ to be equal to the number of paths (in $G$) of length $m-1$ between $x$ and $A_1$.
Note that by definition $w(x)$ is at least equal to the number of vertices $u_i \in A_1$ with $d_G(x,u_i)=m-1$ and by definition of $A_m$ this implies $w(x) \ge 1.$
An equivalent recursive definition of $w$ is the following: we let $w(u_i)=1$ for any $u_i \in A_1$ and for every vertex $x \in A_m$ where $m\ge 2$, we let $$w(x)= \sum_{y \in A_{m-1} \colon yx \in E(G)} w(y).$$
We observe by induction that
\begin{equation}\label{eq:sum_w}
\sum_{ x \in A_m} w(x) \le j(\Delta-1)^{m-1}
\end{equation} for every $1 \le m \le t.$ For $m=1$ this is by definition of $A_1$ and $j$.
For $m\ge 2$, we have by induction that
\begin{align*}
\sum_{ x \in A_m} w(x) &= \sum_{ x \in A_m} \sum_{y \in A_{m-1} \colon yx \in E(G)} w(y)\\
&=\sum_{ y \in A_{m-1}} \sum_{x \in A_m \colon yx \in E(G)} w(y)\\
&\le \sum_{ y \in A_{m-1}} (\Delta-1) w(y)\\
&\le j(\Delta-1)^{m-1}.
\end{align*}

Let $A^{'}_t=\{a \in A_t\mid w(a)<j\}$ and $A^*_t=\{a \in A_t\mid w(a) \ge j\}.$ We first count the edges that are incident to some fixed $a\in A^{'}_t$.
Note that $H$ contains no edges between $a$ and $A_{t+1}$ since for such an edge we would need that $a$ is connected by a path of length $t-1$ to every $u_i, 1 \le i \le j$ and thus in particular we would have $w(a) \ge j.$ By Claim~\ref{cl:no_edge_inNi} we also know that $a$ cannot be incident with an other vertex in $A_t$.
So we only need to count the edges in $H$ between $a$ and  $A_{t-1}$, and by definition of the weight function, this is bounded by $w(a).$

On the other hand, for every $a \in A^*_t$ there are at most $\Delta_H=j \le w(a)$ edges in $E(H)$ incident to $a$.
Having proven that for every  $a \in A_t$ there are at most $w(a)$ edges in $H[A]$ incident with $a$, we conclude (remembering~\eqref{eq:sum_w}) that there are at most $\sum_{ x \in A_t} w(x) \le j(\Delta-1)^{t-1}$ edges in $E(H[A])$ having a vertex in $A_{t}.$
Also we have for every $1 \le m \le t-1$ that the number of edges between $A_{m-1}$ and $A_m$ is bounded by $j(\Delta-1)^{m-1}.$
Hence $$\lvert E(H[A]) \rvert \le \sum_{m=1}^t j(\Delta-1)^{m-1}=\frac{\Delta_H}{\Delta}\lvert E(T_{t,\Delta}) \rvert.$$
Together with~\eqref{eq:EH_R} on the number of edges that intersect $R$, this gives the result as $\deg(v)\le \Delta$ by definition.
\end{proof}

An inspection of the proof yields that the extremal graphs $H$ for Proposition~\ref{sub} satisfy Claim~\ref{cl:no_edge_inNi}, $R=\emptyset$ and for every $x \in A_m$ where $0 \le m \le t-1$, there are exactly $\Delta -1$ edges towards $A_{m+1}.$
Hence such an extremal graph $H$ is exactly $T_{t,\Delta}$ where possibly some of its leaves are identified as one (as long as the maximum degree is still $\Delta$). Let us call such a graph a \emph{quasi-$T_{t,\Delta}$}.\\

Next, we discuss some properties that should be satisfied by any graph that attains the bound of Theorem~\ref{thm:L^tforC_{2t+1}-free} (provided such a graph exists for the given values of $t$ and $\Delta$!).
Let $H \subseteq G$ be a graph such that $E(H)$ is a clique in $L(G)^t$ and $v$ be a vertex of maximum degree in $H$, which maximises $|E(H)|$ among all choices for $G$ and $H$.
Let $N_H(v)=\{u_1, \ldots, u_j\}$. Then in particular, in $L(G)^t$  every edge of $H$ is adjacent to every edge $vu_i$, for all $1\le i \le j$. 

So by Proposition~\ref{sub}, for every vertex $v$ of degree $\Delta$
we observe locally a quasi-$T_{t, \Delta}$ again, and in particular every neighbour of such a $v$ has degree $\Delta$ (for $t \ge 2$). 
So $H$ is $\Delta$-regular and in particular a connected component of $G$. So it is not a tree and hence has some girth.
The girth is at least $2t$ (as for every vertex we locally have a quasi-$T_{t, \Delta}$), but it cannot be $2t+1$ since $G$ is $C_{2t+1}$-free and it cannot be $2t+2$ or more since $E(H)$ is a clique in $L(G)^t$. 
Also we observe that for every $a \in A_t^*$ the condition that $w(a)\ge j$ implies that it has $\Delta$ neighbours in $A_{t-1}$ as these all have a weight function equal to $1$ and so it has no neighbours in $A_{t+1}.$
Hence $H$ is a $\Delta$-regular graph with girth $2t$ and diameter $t$.
In particular they need to be Moore graphs and consequently by~\cite{Sing66, BI73} the extremal graphs are polygons when $t \ge 3.$

\section{A general bound on $\omega(L(G)^t)$}\label{sec:general}

When $H\subseteq G$ is a graph whose edges form a clique in $L(G)^t$, it implies in particular that all edges adjacent to a specific vertex $v$ are at distance at most $t-1$ from all other edges. 
As $\lvert E(T_{t,\Delta}) \rvert \le \Delta^t$, the following proposition implies Theorem~\ref{thm:mainclique}.

\begin{prop}\label{substrong}
	For fixed $\Delta$ and $t$, let $G$ be a graph with maximum degree $\Delta$ and
	$H\subseteq G$ be a subgraph of $G$ with maximum degree $\Delta_H$.
	Let $v$ be a vertex with degree $d_H(v)=\Delta_H=j$ and let $u_1,u_2, \ldots, u_j$ be its neighbours. Suppose that in $L(G)^t$,  every edge of $H$ is adjacent to  $vu_i$ for every $1 \le i \le j$.
	Then $\lvert E(H) \rvert \le \frac{3}{2}\lvert E(T_{t,\Delta}) \rvert$.	
\end{prop}

\begin{proof}
	We do this analogously to the proof of Proposition~\ref{sub}.
	For fixed $\Delta$, let $H$ and $G$ be graphs satisfying all conditions, such that $\lvert E(H) \rvert$ is maximized (which is again possible since $j\Delta ^t$ is an upper bound for $\lvert E(H) \rvert$). 

	It suffices to show that $\lvert E(H) \rvert \le \frac{3}{2} \lvert E(T_{t,\Delta}) \rvert$. 

	By the proof of Claim~\ref{cl:no_edge_inNi} we know that for any $1 \le i \le t-1$, the set $N_i$ does not induce any edges of $H$ (but this is not necessarily true anymore for $N_t$).	
	
	Define $A_m, R_m$, the weight function $w$, $A^{'}_t$ and $A_t^*$ as has been done in the proof of Proposition~\ref{sub}.

As before, the number of edges that (are not induced by $N_t$ and) use at least one vertex of $R$ is bounded by~\eqref{eq:EH_R}. Also, we again have for every $1 \le m \le t-1$ that the number of edges between $A_{m-1}$ and $A_m$ is bounded by $j(\Delta-1)^{m-1}.$  Furthermore, $R_t$ does not induce any edge of $H$, because such an  edge would be at distance larger than $t$ from $vu_1$. Thus the number of edges of $H$ that are either disjoint from $A_t$, or join $A_t$ and $R \backslash R_{t}$, is at most

\begin{equation}\label{eq:EH_bla}
(\Delta-j)\left(\frac1{\Delta}\lvert E(T_{t,\Delta}) \rvert -1 \right) + \sum_{m=1}^{t-1} j(\Delta-1)^{m-1}.
	\end{equation}

We will derive that the remaining edges of $H$ (which all intersect $A_t$) can be bounded by a linear combination of the weight functions $w(a)$ of the vertices $a\in A_t$.\\

	For every $a \in A_t^*$ there are at most $j \le w(a)$ edges in $E(H)$ having $a$ as one of its endvertices. So let us now focus on the edges that intersect $A^{'}_t$.
	
	We observe that there are no edges in $H$ between any $a \in A^{'}_t$ and $r \in R_t$, because there is some $u_i$ such that $d(a,u_i) \ge t$, which implies that $vu_i$ and $ar$ would be at distance larger than $t$. For the same reason $H$ has no edges between $A^{'}_t$ and $A_{t+1}$.\\
	
	Finally, we want to count the edges between $A_{t-1}$ and $A^{'}_t$, as well as those that are induced by $A^{'}_t$.
	We will prove that their number is bounded by $\frac 32 \sum_{a \in A^{'}_t} w(a).$

For that, we need the following technical claim.

	\begin{claim}\label{claim:technical}
		Let $j$ be fixed and assume $j>x \ge m>0$ and $j>y \ge n>0$ with $x+y \ge j$. Then 
		$$\frac{ \frac32 x -m}{j-m}+\frac{ \frac32 y -n}{j-n} \ge 1.$$
		Equality occurs if and only $m=n=x=y=\frac j2.$
	\end{claim}
	\begin{claimproof}
		Multiplying both sides with the positive factor $2(j-m)(j-n)$, we need to prove that $3(x+y)j-3xn-3ym+2mn \ge 2j^2$.
	For fixed $j,x$ and $y$ the left hand side is minimal when $m=x$ and $n=y$. This reduces to proving that $3(x+y)j -4xy \geq 2j^2.$
But this is true since
	\begin{align*}
	3(x+y)j -4xy-2j^2 &= 0.25 j^2 - (x+y-1.5j)^2 + (x-y)^2\\
	&=(2j-(x+y))\cdot(x+y-j)+(x-y)^2\\&\ge 0,
	\end{align*}
	as $j \le x+y <2j$, i.e.~$|x+y-1.5j| \le 0.5j$.
	\end{claimproof}

For every $a\in A_t^{'}$, let $m(a)$ denote the number of neighbours (in $H$) of $a$ in $A_{t-1}$ and let $q(a)$ denote the number of neighbours (in $H$) of $a$ in $A^{'}_t$.
Furthermore, we define $f(a)=\frac{ \frac32 w(a) -m(a)}{j-m(a)}$.

Suppose $H$ contains an edge $e$ between two vertices $a_1, a_2 \in A^{'}_t$. Then $w(a_1)+w(a_2) \ge j$, since $a_1a_2$ must be within distance $t-1$ of each of $vu_1,vu_2,\ldots,vu_j$. Hence by Claim~\ref{claim:technical}  (applied with $m=m(a_1)$, $n=m(a_2)$, $x=w(a_1)$ and $y=w(a_2)$), we obtain that $f(a_1)+f(a_2)\geq 1$ for every edge $a_1a_2$ of $H[A^{'}_t]$. 

From this it follows that $|E(H[A^{'}_t])|  \leq \sum_{a_1a_2 \in E(H[A^{'}_t])} f(a_1)+f(a_2)$. The right hand side can further be rewritten as $ \sum_{a \in A^{'}_t} q(a) \cdot f(a)$. Since every vertex $a \in A^{'}_t$ has $q(a) \leq j-m(a)$ neighbours in $A^{'}_t$ and has $m(a)$ neighbours in $A_{t-1}$, we conclude that
the number of edges of $H$ that are either induced by $A^{'}_t$ or join  $A^{'}_t$ and $A_{t-1}$ is at most

\begin{align*}
\sum_{a \in A^{'}_t} \left((j-m(a)) \cdot f(a) + m(a) \right) &=  \sum_{a \in A^{'}_t} \frac{3}{2}  w(a).
\end{align*}

Thus the number of edges in $E(H)$ using
at least one vertex in $A_t$ is bounded by $$\sum_{ a \in A^{'}_t} \frac32 w(a) +\sum_{ a \in A^{*}_t} w(a)  \le \sum_{ x \in A_t} \frac32 w(x),$$
which (see the derivation of~\eqref{eq:sum_w}) is at most $\frac{3}{2} j (\Delta-1)^{t-1}$.
Summing this and~\eqref{eq:EH_bla}, we conclude that $H$ has fewer than $(\Delta-j)\left(\frac1{\Delta}\lvert E(T_{t,\Delta}) \rvert -1 \right) + \frac{3}{2}\frac{j}{\Delta} |E(T_{t,\Delta})| \leq \frac{3}{2} |E(T_{t,\Delta})|$ edges.	
\end{proof}

Note that the exact maximum in Proposition~\ref{substrong} is $\sum_{m=1}^{t-1} \Delta( \Delta-1)^{m-1}+\frac32 \Delta( \Delta-1)^{t-1}$ and this can be attained when $\Delta$ is even.  
For example when $t=2$, the following example in Figure~\ref{fig:beatC5} shows that the blow-up of a $C_5$ is not extremal anymore when only taking into account the weaker conditions from Proposition~\ref{substrong}.
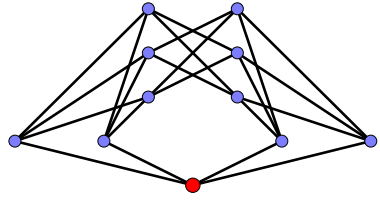
\begin{figure}[h]
	\centering
	\scalebox{0.9}{
		\rotatebox{90}{
			\begin{tikzpicture}[line cap=round,line join=roundx,x=1.3cm,y=1.3cm]
			\clip(1.4,-2.1) rectangle (3.6,2.5);
			\draw [line width=1.1pt] (1.5,0.) -- (2.,2);
			\draw [line width=1.1pt] (1.5,0.) -- (2.,1);
			\draw [line width=1.1pt] (1.5,0.) -- (2.,-1);	
			\draw [line width=1.1pt] (1.5,0.) -- (2.,-2);
			
			\draw [line width=1.1pt](3,-0.5)--(2.5,0.5);
			\draw [line width=1.1pt](3,-0.5)--(3.5,0.5);
			\draw [line width=1.1pt](2.5,-0.5)--(3,0.5);
			\draw [line width=1.1pt](2.5,-0.5)--(3.5,0.5);
			\draw [line width=1.1pt](3.5,-0.5)--(3,0.5);
			\draw [line width=1.1pt](3.5,-0.5)--(2.5,0.5);

			\draw [line width=1.1pt](2.,2)--(2.5,0.5);
			\draw [line width=1.1pt](2.,2)--(3.5,0.5);
			\draw [line width=1.1pt](2.,2)--(3,0.5);
			\draw [line width=1.1pt](2.,1)--(3.5,0.5);
			\draw [line width=1.1pt](2.,1)--(3,0.5);
			\draw [line width=1.1pt](2.,1)--(2.5,0.5);
			
			\draw [line width=1.1pt](3,-0.5)--(2.,-1);
			\draw [line width=1.1pt](3.5,-0.5)--(2.,-1);
			\draw [line width=1.1pt](2.5,-0.5)--(2.,-1);
			\draw [line width=1.1pt](2.5,-0.5)--(2.,-2);
			\draw [line width=1.1pt](3,-0.5)--(2.,-2);
			\draw [line width=1.1pt](3.5,-0.5)--(2.,-2);

			\draw [fill=red] (1.5,0.) circle (3pt);
			\draw [fill=xdxdff] (2.,1) circle (2.5pt);
			\draw [fill=xdxdff] (2.,-2.) circle (2.5pt);
			\draw [fill=xdxdff] (2.,-1) circle (2.5pt);
			\draw [fill=xdxdff] (2.,2.) circle (2.5pt);

			\draw [fill=xdxdff] (3,-0.5) circle (2.5pt);
			\draw [fill=xdxdff] (2.5,-0.5) circle (2.5pt);	
			\draw [fill=xdxdff] (3.5,-0.5) circle (2.5pt);

			\draw [fill=xdxdff] (3,0.5) circle (2.5pt);
			\draw [fill=xdxdff] (2.5,0.5) circle (2.5pt);	
			\draw [fill=xdxdff] (3.5,0.5) circle (2.5pt);
			\end{tikzpicture}}}
	\caption{An extremal graph for Proposition~\ref{substrong} for $\Delta=4$ and $t=2$.}
	\label{fig:beatC5}
\end{figure}

\section{Determination of $h_3(3)$}\label{sec:subcubic}

\begin{proof}[Proof of Theorem~\ref{thm:cubic}]
	Let $G=(V,E)$ be a graph of maximum degree $3$ such that the line graph $L(G)$ of $G$ has diameter at most $3$, i.e.~$L(G)^3$ is a clique. If we can show that $G$ must have at most $22$ edges, then the result is proven. Suppose to the contrary that $|E|\ge23$. The proof proceeds through a series of claims that establish structural properties of $G$.
	
	In each claim, we will estimate $|E|$ by performing a breadth-first search rooted at some specified edge $e$ up to distance $3$. To avoid repetition, let us set out the notation we use each time. We write $e=uv$. Let $u_0$ and $u_1$ be the two neighbours of $u$ other than $v$ (if $u$ has degree $3$). For $i\in\{0,1\}$, let $u_{i0}$ and $u_{i1}$ be the two neighbours of $u_i$ other than $u$ (if $u_i$ has degree $3$). For $i,j\in\{0,1\}$, let $u_{ij0}$ and $u_{ij1}$ be the two neighbours of $u_{ij}$ other than $u_{i}$ (if $u_{ij}$ has degree $3$). Similarly, define $v_i$, $v_{ij}$, $v_{ijk}$ for $i,j,k\in\{0,1\}$.
	
	\begin{claim}\label{clm:triangle}
		$G$ contains no triangle, loop or multi-edge.
	\end{claim}
	
	\begin{claimproof}
		These $3$ cases are straightforwardly bounded by the breadth-first search.
		If the edge $e$ is in a triangle, $|E|=|N_{L(G)^3}[e]| \le |E(K_3)|+2\cdot|E(T^1_{3,3})|+|E(T^1_{2,3})|= 3+2\cdot7+3= 20$.
		Analogously, if the edge $e$ is a loop, one obtains $|E|=|N_{L(G)^3}[e]| \le 1+7=8$.
		If the edge $e$ has a parallel edge then $|E|=|N_{L(G)^3}[e]| \le 2+2\cdot7= 16$.
	\end{claimproof}

\begin{claim}\label{clm:regular}
	$G$ is $3$-regular, and so $|E|$ is divisible by $3$.
\end{claim}

\begin{claimproof}
	If not, say, $v$ has degree at most $2$, then, say, $v_1$, $v_{1j}$, $v_{1jk}$ are undefined, and so $|E|=|N_{L(G)^3}[e]| \le 1+3\cdot7=22$, a contradiction.
\end{claimproof}

\begin{claim}\label{clm:C4}
$G$ contains no $4$-cycle.
\end{claim}

\begin{claimproof}
If the edge $e$ is in a $4$-cycle, then without loss of generality suppose $u_1 =v_{00}$, $v_0=u_{11}$, and so on. Already $|E|=|N_{L(G)^3}[e]| \le 4+2\cdot7+2\cdot3=24$ and by Claim~\ref{clm:regular} we have a contradiction if we can show that $|E|$ is $1$ lower. So we may assume that $u_0$, $u_1$, $v_0$, $v_1$, $u_{00}$, $u_{01}$, $u_{10}$, $v_{01}$, $v_{10}$, $v_{11}$ are all distinct vertices and that the vertices $u_{00k}$, $u_{01k}$, $u_{10k}$, $v_{01k}$, $v_{10k}$, $v_{11k}$ (possibly not all distinct) are all at distance exactly $3$ from $e$.

Consider the edges $u_{00}u_{000}$, $u_{00}u_{001}$, $u_{01}u_{010}$ and $u_{01}u_{011}$. They are within distance $3$ (in $L(G)$) from $vv_0$, so $u_{000}$, $u_{001}$, $u_{010}$ and $u_{011}$ all need to be adjacent to $v_{01}$, leading to a contradiction as $\deg{v_{01}}\le 3.$
\end{claimproof}

\begin{claim}\label{clm:C5}
$G$ contains no $5$-cycle.
\end{claim}

\begin{claimproof}
If the edge $e$ is in a $5$-cycle, then without loss of generality suppose $u_{11} =v_{00}$, $u_{111}=v_{000}$ and so on. Already $|E|=|N_{L(G)^3}[e]| \le 5+2\cdot 7+ 2\cdot 3+1= 26$. 
 Since $26 \ge |E|\geq 23$ and $G$ is $3$-regular by Claim~\ref{clm:regular}, it follows that $|E|=24$ and $|V|=   \frac{2 |E|}{3} = 16$.

Note first that $N_2(e)=\{u_{00},u_{01},u_{10},u_{11},v_{01},v_{10},v_{11}\}$ are all distinct vertices or else $|E|$ is already at most $23$. 
Thus $|N_3(e)|=16-13=3$ and so (again using $3$-regularity, and also the fact that $N_3(e)$ must be an independent set) there are exactly $2$ edges in the subgraph induced by $N_2(e)$.

We divide our considerations into two cases. First, we assume $u_{11}$ has a neighbour in $N_2(e).$ By Claim~\ref{clm:triangle}, without loss of generality we can assume that this neighbour is $u_{00}$ and thus $u_{111} =v_{000}=u_{00}$, $u_{11}=u_{001}$ and so on.
Since $N_2(e)$ induces two edges, we can assume that $v_{100}, v_{110} \in N_3(e)$.
Note that the edge $u_{00}u_{11}$ is within distance $3$ (in $L(G)$) of both $v_{10}v_{100}$ and $v_{11}v_{110}$. 
It cannot be that $u_{000}$ is equal to $v_{10}$ or $v_{11}$ or else one of the edges $v_{10}v_{100}$ and $v_{11}v_{110}$ remains too far from $u_{00}u_{11}$ (taking Claim~\ref{clm:C4} into account). 
At this point, we note that $v_{10}$ and $v_{11}$ both need to be adjacent to $u_{000}$, creating a $C_4$ and hence leading to a contradiction.

Second, since we are not in the first case, $u_{111}=v_{000}$ must be at distance exactly $3$ from $e$. 
The vertex $u_{111}$ must have all of its $3$ neighbours in $N_2(e)$, one of which is $u_{11}$. Keeping in mind that there is no four-cycle, 
 we can therefore assume without loss of generality that $u_{111}=u_{000}=v_{111}$.

Let us consider as a subcase the possibility that $u_{01}$ and $v_{11}$ are adjacent (the case $v_{10}$ and $u_{00}$ being adjacent, is done in exactly the same way), say, $u_{010}=v_{11}$. Note that the edge $u_{01}v_{11}$ is within distance $3$ (in $L(G)$) of both $u_1u_{10}$ and $v_0v_{01}$. Since $v_{11}$ has all its neighbours already fixed (and keeping in mind that $N_2(e)$ induces only one edge other than $u_{01}v_{11}$), it can only be that $u_{01},u_{10}$ and $v_{01}$ have a common neighbour in $N_3(e)$.
So without loss of generality, $u_{011}=u_{100}=v_{010}$. But now, with only the free placement of $u_{001}$, the only possibility to have $u_{00}u_{000}$ within distance $3$ (in $L(G)$) of both edges $u_{10}u_{100}$ and $v_{01}v_{010}$, is if $u_{001}$ is equal to $u_{10}$ or $v_{01}$.
But then we have already determined the two edges induced by $N_2(e)$, none of which is incident to $v_{10}$, so that both $v_{100}$ and $v_{101}$ must be in $N_3(e)$, leading to the contradiction that $|N_3(e)|=|\left\{u_{100},u_{111},v_{100},v_{101} \right\}| \geq 4$.

We have thus shown that $u_{01}v_{11}$ and $v_{10}u_{00}$ are not present as an edge.

Let $i\in \left\{0,1\right\}$. The vertex $u_{01}$ is not adjacent to any vertex in $\{u_{00},u_1,v_0,v_{11}\}$ and so $u_{01i}$ has to be adjacent to one of them to ensure that $u_{11}u_{111}$ is within distance $3$ (in $L(G)$) of $u_{01}u_{01i}$

This implies $u_{01i}$ has to be equal to $v_{110}$, $u_{10}$, $v_{01}$ or $v_{10}$ 
 (taking Claims~\ref{clm:triangle} and~\ref{clm:C4} into account).
Incidentally, $u_{01i}$ can also not be equal to $v_{10}$, because in order for $u_{01}v_{10}$ to be within distance $3$ of $u_{01}u_{01i}$, we would need an edge between $\left\{u_{01},v_{10}  \right\}$ and $\left\{u_{00},v_{11}\right\}$, which would either create a triangle or an edge that we already showed to be not present.

So $u_{01i}$ has to be equal to $v_{110}$, $u_{10}$ or $v_{01}$, and symmetrically $v_{10i}$ equals $u_{001}, u_{10}$  or $v_{01}$, for all $i\in \left\{0,1\right\}$. 
As there are only two edges in the graph induced by $N_2(e)$, and both $u_{01}$ and $v_{10}$ are an endvertex of one of them,
we may conclude without loss of generality that $u_{010}=v_{110}$ and $v_{101}=u_{001}$.
Note that the edge $v_{11}v_{110}$ is within distance $3$ (in $L(G)$) of $uu_1$, and so $u_{10}v_{110}$ must be an edge. But then the distance between $u_{01}v_{110}$ or $u_{10}v_{110}$ and $vv_0$ is at least $4$, a contradiction.
\end{claimproof}

By the above claims, it only remains to consider $G$ being $3$-regular and of girth at least $6$. Let $e\in E$ be arbitrary. Then we have $|E|=|N_{L(G)^3}[e]| \le 29$ and by Claim~\ref{clm:regular} we have a contradiction if we can show that $|E|$ is $6$ lower.
Since $|E|\geq 23$ and $G$ is $3$-regular, we know  $|V|=\left \lceil  \frac{2 |E|}{3} \right \rceil \geq 16$, hence there are at least $16 - (2+4+8) = 2$ vertices at distance $3$ from $e$.
Let $x$ and $y$ be vertices at distance $3$ from $e$. We may assume without loss of generality that $x$ is adjacent to $u_{00}$, $u_{10}$, and $v_{00}$. Since the edge $vv_1$ is within distance $3$ (in $L(G)$) of both edges $u_{00}x$ and $u_{10}x$, it follows (without loss of generality) that $u_{00}v_{10}$ and $u_{10}v_{11}$ are edges.

Since $y$ must satisfy similar constraints as $x$, and it cannot be adjacent to $u_{00}$ nor to $u_{10}$, there will be at least three edges between vertices in $N_2(e)$ and so similarly as before, we know that  $|V|=\left \lfloor\frac{2|E|}{3}\right \rfloor \leq \left \lfloor\frac{2 \cdot (29-3)}{3}\right \rfloor =17$. Because every $3$-regular graph has an even number of vertices, it follows that $|V|=16$, so that in fact $x$ and $y$ are the only vertices in $N_{3}(e)$. From this and $3$-regularity, we conclude that the subgraph induced by $N_2(e)$ must have exactly $5$ edges.
Since every edge between $2$ vertices in $N_2(e)$ will be between some $u_{ij}$ and a $v_{k\ell}$ and $G$ is $3$-regular, we know that $y$ is adjacent to exactly one of $u_{01}$ and $u_{11}$, wlog $u_{11}$.
Similarly the two neighbours of $y$ of the form $v_{ij}$ are not a neighbour of $x$ and so $N(y)$ and $N(x)$ are disjoint.
In particular $y$ cannot be adjacent to $v_{00}$ and so it has to be adjacent to $v_{01}.$
The last neighbour of $y$ is either $v_{10}$ or $v_{11}$. If it is $v_{11}$, then to ensure that $uu_0$ is within distance $3$ of both $yv_{01}$ and $yv_{11}$, we would need $u_{01}$ to be a neighbour of both $v_{10}$ and $v_{11}$, creating a four-cycle; contradiction. Thus the neighbours of $y$ must be $u_{11}, v_{01}$ and $v_{10}$.

So to ensure this, $u_{01}v_{01}$ is an edge as well.

We have now determined the whole graph, apart from two edges between $\left\{u_{01},u_{11}\right\}$ and $\left\{v_{00},v_{11} \right\}$. However, the edge $v_{00}u_{11}$ would create the five-cycle $xu_{10}u_1u_{11}v_{00}$, while the edge $v_{00}u_{01}$ would yield the four-cycle $u_{01}v_{00}v_{0}v_{01}$. So in both cases, we get a contradiction, from which we conclude.
\end{proof}

A brief inspection of the proof in Claim~\ref{clm:regular} yields that the extremal graph has exactly one vertex of degree $2$ and $14$ vertices of degree $3$.
Let the vertex of degree $2$ be $w$. Let its two neighbours be $u$ and $v$ and then $u_i,v_i,u_{ij},v_{ij}$ for $i,j \in \{0,1\}$ are defined as before.
Noting that every $u_{ij}$ has two neighbours of the form $v_{0k}$ and $v_{1 \ell}$ where $k, \ell \in \{0,1\}$, one can check that there is a unique extremal example with respect to Theorem~\ref{thm:cubic}, namely, the point--line incidence graph of the Fano plane, in which exactly one edge is subdivided.

\subsection*{Acknowledgement}

We are grateful to the anonymous referees for their helpful comments and suggestions.

\bibliographystyle{abbrv}
\bibliography{edgediam}

\end{document}